\documentclass[oneside]{amsart} 
\usepackage{amssymb}
\usepackage{hyperref}
\usepackage[a4paper]{geometry}
\usepackage{dsfont} 

\theoremstyle{plain}
\newtheorem{thm}{Theorem}[section] 
\newtheorem{prop}[thm]{Proposition} 
\newtheorem{lem}[thm]{Lemma} 
\newtheorem{cor}[thm]{Corollary} 
\theoremstyle{remark} 
\newtheorem*{rem}{Remark}
\theoremstyle{definition}

\numberwithin{equation}{section}

\renewcommand*{\div}{\operatorname{div}}

\newcommand*{\id}{\operatorname{Id}}

\newcommand*{\bydef}{\overset{\rm def}{=}}
\newcommand*{\norm}[1]{\left\Vert #1\right\Vert}

\begin{document}

\title[Ill-posedness of Euler--Maxwell]{Ill-Posedness of the Incompressible Euler--Maxwell Equations in the Yudovich Class}

\author{Haroune Houamed}
\address{New York University Abu Dhabi \\
Abu Dhabi \\
United Arab Emirates} 
\email{\href{mailto:haroune.houamed@nyu.edu}{haroune.houamed@nyu.edu}}

\thanks{\emph{Acknowledgements.} The author extends his sincere thanks to Diogo Ars\'enio for many fruitful discussions on plasma models.}

\keywords{perfect incompressible two-dimensional fluids, Maxwell's system, plasmas, Yudovich  theory, ill-posedness}
\date{\today}

\begin{abstract}
It was shown   recently by Ars\'enio and the author that the two-dimensional incompressible Euler--Maxwell system is globally well-posed in the Yudovich class, provided that the electromagnetic field enjoys appropriate conditions, including the  {\em Normal Structure}. In this paper, we prove that this assumption is sharp, in the sense that the Euler--Maxwell system becomes ill-posed in the Yudovich class for initial data that do not obey the   {\em Normal Structure} condition. The proof applies to both the whole plane and the two-dimensional torus, and holds for any value of the speed of light $c\in (0,\infty)$. This is achieved by expanding the magnetic field around a horizontal background and showing that the Lorentz force can be decomposed into two parts: the first is in the form of a singular operator acting on the vorticity, and the second, a ``remainder'', is of lower order when analyzed in a specific time regime.
\end{abstract}

\maketitle

\tableofcontents


\section{Introduction}

\subsection{The model of interest}

We consider the incompressible Euler--Maxwell system
\begin{equation}\label{EM}\tag{EM}
	\begin{cases}
		\begin{aligned}
			\text{\tiny(Euler's equation)}&&&\partial_t u +u \cdot\nabla u = - \nabla p + j \times B, \qquad &\div u =0,&
			\\
			\text{\tiny(Amp\`ere's equation)}&&&\frac{1}{c} \partial_t E - \nabla \times B =- j , 
			\\
			\text{\tiny(Faraday's equation)}&&&\frac{1}{c} \partial_t B + \nabla \times E  = 0 , &\div B = 0,&
			\\
			\text{\tiny(Ohm's law)}&&&j= \sigma \big( cE +  u \times B\big) , &
		\end{aligned}
	\end{cases}
\end{equation}
 set in the whole two-dimensional plane or torus, i.e., $ t\in (0,\infty )$, and $x\in  \mathbb{R}^2$ or $x\in  \mathbb{R}^2/2\pi \mathbb Z$, supplemented with the initial data $(u,E,B)|_{t=0}=(u_0,E_0,B_0)$.  In this setting, the principle  unknowns $u$, $E$ and $B$ correspond to the velocity of the conducting fluid and its  electromagnetic field, respectively. The second type of unknowns, i.e., the scalar pressure $p$ and the current density $j$, can be recovered via solving an elliptic equation, and by virtue of Ohm's law, respectively. Finally, the non-negative constants $c,\sigma \in (0,\infty)$ refer to the speed of light and the electrical conductivity parameter, respectively. 
 
 We refer to \cite{bis-book, D-book} for more detailed discussions about the physical background of this system of equations, and to \cite{as} for the derivation of various plasma models from the Vlasov--Maxwell--Boltzmann system.  We further refer to \cite{a19, ag20, AHB24, aim15, gim, IK2011, MN, RW24, FW22} for more advances on the analysis of the viscous version of these equations corresponding to the Navier--Stokes--Maxwell system.
 
\subsection{Motivation and Main result} As it is apparent in its formulation above, the Euler--Maxwell model arises from the nonlinear coupling of the Euler equations with the full set of Maxwell equations governing electromagnetism. Here, we are interested in the question of   extending  Yudovich  theory \cite{Y63} to \eqref{EM}. More specifically, we address the following question:
 
{\em \textbf{Question.}   Can the contribution of the Lorentz force $j\times B$ in \eqref{EM} be treated as a perturbation, thus allowing Yudovich theory to hold if the electromagnetic field is ``sufficiently well-behaved''}

 This question was positively answered in \cite{ah}, provided the following three conditions on the initial data are satisfied:
 
 \begin{enumerate}
 	\item { \em Normal Structure.}  The electromagnetic field acts in ``good'' directions,  meaning that  	\begin{equation*} 
	E(0,x)=
	\begin{pmatrix}
		E_1(0,x)\\E_2(0,x)\\0
	\end{pmatrix}
	\qquad\text{and}\qquad
	B(0,x)=
	\begin{pmatrix}
		0\\0\\b(0,x)
	\end{pmatrix}.
\end{equation*}
 	This structure is   preserved by the system \eqref{EM} for all positive times.
 	\item {\em Regularity.} The electromagnetic field is initially regular enough, in the sense that 
 	\begin{equation*}
 		(E_0,B_0) \in B^\frac{7}{4} _{2,1}(\mathbb R^2).
 	\end{equation*}
 	\item { \em Non-relativistic regime.} The speed of light is sufficiently large compared to the size of the initial data (including the fluid velocity) in certain suitable functional spaces, i.e.,
 	\begin{equation*}
 		 c_0 <c
 	\end{equation*}
 	where $c_0$ depends in a nonlinear way on the size of the initial data.
 \end{enumerate}
 
 A very natural question that arises is whether the preceding three conditions are optimal and necessary for extending Yudovich's theory to the Euler--Maxwell system   \eqref{EM}. In this paper, we prove that  the {\em Normal Structure} is indeed sharp. More specifically, we show that \eqref{EM} is {\em mildly ill-posed} in the Yudovich  class if the {\em Normal Structure} is not satisfied. This is the content of following statement, which represents  the main result of the paper.
 
 \begin{thm}\label{main:thm}
 	Let $\sigma,c\in (0,\infty)$, and consider \eqref{EM} on $\mathbb R^2$ or $\mathbb T^2$. There is a universal constant $c_*>0$, depending only on $c$ and $\sigma$, such that, for any $N>0$, there is   a smooth initial datum $(u_{0,N}, E_{0,N},B_{0,N})$ satisfying
 	\begin{equation*}
 		B_{0,N} \cdot \vec e_1 \neq 0  \quad \text{and/or} \quad B_{0,N} \cdot \vec e_2 \neq 0,
 	\end{equation*}
 	with 
 	\begin{equation*}
 		\norm {\omega_{0,N}}_{L^\infty} =\mathcal O \left(\frac{1}{N}\right),
 	\end{equation*}
 	and generating a unique local-in-time smooth solution $ (u_N,E_N,B_N)$ of \eqref{EM} such that 
 	\begin{equation*}
 		\norm {\omega_N (t)}_{L^\infty} \geq c_*,
 	\end{equation*}
 	for some positive time $t \lesssim  1/N$.
 \end{thm}
 
 Theorem \ref{main:thm} establishes that \eqref{EM} is mildly ill-posed in the Yudovich  class for any finite value of the speed of light $c\in (0,\infty)$. The formal singular limit  $c=\infty$ corresponds to the MHD model (see \eqref{MHD}, below), which has been similarly shown to be mildly ill-posed in  \cite{WZ23}.  On the other hand, the (non-physical) case $c=0$  corresponds to the Euler equation with a forcing term given by a Riesz-type operator in the vorticity formulation, a model that is also ill-posed (this will be discussed further later on). To summarize, our Theorem \ref{main:thm} demonstrates that the main source of ill-posedness across all these configurations is the ``bad'' direction of the electromagnetic field, which disrupts Yudovich's theory in this context. This concludes the discussion on the mild ill-posedness (or, equivalently, the discontinuity of the solution map) for all these fluid-plasma models in the Yudovich class. 
 
 We conclude by pointing out that it would be particularly interesting to strengthen the mild ill-posedness to strong ill-posedness for both the Euler--Maxwell system  \eqref{EM} and the MHD model. However, achieving this requires more refined analysis, which lies beyond the scope of this paper. Indeed, it is not straightforward to extend the recent results of Elgindi and Shikh Khalil  \cite{EK25}, who established strong ill-posedness for the Euler--Riesz system in the Yudovich class, to the more complex systems \eqref{EM} and MHD. This remains a challenging open problem, which we hope to address in future work.

\subsection{Heuristic discussion}

The main  finding of this paper establishes that \eqref{EM} is mildly ill-posed in the Yudovich  class. Specifically, we demonstrate that the ill-posedness arises when solving the equations around a steady magnetic background. Broadly speaking, the source of ill-posedness in this case --- solving the equations near a particular steady state --- shares some similarities with the approach used in \cite{WZ23} for the Magneto-Hydrodynamic (MHD) model:
\begin{equation}\label{MHD} \tag{MHD}
	\begin{cases}
		\begin{aligned}
			\partial_t u + u\cdot \nabla u + \nabla p= (\nabla \times  B)\times B, \qquad \div u &=0
		\\
		\partial_t B -\Delta B = \nabla \times (u\times B), \qquad \div B &=0.
		\end{aligned}
	\end{cases}
\end{equation}
However,   the core of our analysis is fundamentally different due to the parabolic-hyperbolic nature of Maxwell’s equations and the more intricate structure of the Lorentz force in the velocity equation.

For clarity, we note that    \eqref{MHD}  can be derived from the Euler--Maxwell system \eqref{EM} by taking the non-relativistic limit $c\to \infty$. This formal procedure can, in fact, be rigorously justified in various contexts, see \cite{ahh24, ah2, JKK25} for more details on this matter. 

Finally, we refer the interested reader to \cite{JS22, JS24} for related ill-posedness results for another version of plasma models involving the Hall-term.

In the case of the Euler--Maxwell system \eqref{EM}, the source of ill-posedness can be more clearly observed (in simpler terms) by considering the (non-physical) situation where $c\to 0$. Formally,  as $c\to 0$, the Euler--Maxwell system \eqref{EM} reduces to the following perturbed Euler equations when recast in vorticity form:
\begin{equation}\label{Toy:model}\tag{Toy Model}
	\partial_t \bar \omega + \bar  u \cdot \nabla \bar  \omega =  \nabla \times \Big ( (\bar  u\times B_0 )\times B_0 \Big ), \qquad \bar u =  \nabla^\perp \Delta^{-1} \bar \omega,
\end{equation} 
where we formally define 
\begin{equation*}
	(\bar u , \bar  \omega) \bydef  \lim_{c\to 0} (u,\omega).
\end{equation*}
Thus, by choosing an initial magnetic field of the form
\begin{equation*}
	B_0= (\alpha,\beta,0) ,
\end{equation*}
for any fixed $(\alpha,\beta) \in \mathbb R^2\setminus \{ (0,0)\}$, we obtain the following configuration (assuming  without loss of generality   that $\alpha\neq 0$)  
\begin{equation}\label{Euler-Riesz} \tag{ER}
	\partial_t \bar\omega + \bar u \cdot \nabla \bar\omega = \beta^2  \bar\omega +  (\alpha^2+\beta^2) \mathcal R \bar\omega, 
\end{equation}
where we set 
\begin{equation}\label{Riesz:def}
	\mathcal R  \bydef \partial_{x_1}^2  (-\Delta )^{-1}.
\end{equation}

When $\beta=0$, Elgindi and Masmoudi \cite{EM20} established first that the Euler--Riesz model \eqref{Euler-Riesz}, for $\beta=0$, is mildly ill-posed in the Yudovich class. The same proof would work with $\beta\neq 0$ as well. The key observation in their proof is that Riesz operators are not continuous on $L^\infty$, which initially suggests   that the source term in \eqref{Euler-Riesz} could prevent the propagation of the initial $L^\infty$ norm of the vorticity. This observation was rigorously justified by employing two essential ingredients: First, if one formally drops the transport term in \eqref{Euler-Riesz}, expanding the solution around $t=0$ leads directly to (strong) ill-posedness due to the continuity defect of the Riesz operator $\mathcal R$ in  $L^\infty$. The second ingredient  involves rewriting the solution along the characteristic, thereby removing the transport term from the equations. However, this requires dealing with the commutator between Riesz's operator and the velocity field's flow. To close this argument, a crucial commutator estimate was   obtained in \cite{EM20}  (see Lemma \ref{lemma:commutator:Riesz}, below), which will   be a central element  in our analysis here, too.

More recently, Elgindi and Shikh Khalil \cite{EK25} made progress on the model \eqref{Euler-Riesz} (with $\beta=0$), establishing a strong ill-posedness result in the sense of norm inflation. Their approach is based on deriving a nonlinear leading order system for the Euler equations with Riesz forcing, which turns out to be strongly ill-posed in $L^\infty$. The analysis laid out in  \cite{EK25} also suggests that the transport term counteracts the linear growth, causing the full nonlinear equation to exhibit growth of a lower order of magnitude than the linear one.   Once again, extending this mechanism  to \eqref{EM} and \eqref{MHD} would be of a great interest. However, the validity of the techniques from \cite{EK25} for more complex perturbations of Euler equations, such as Euler--Maxwell and MHD, requires more careful analysis that goes beyond the scope of this paper. 

\subsection{Roadmap to ill-posedness   in the Yudovich  class}

 Given the preceding discussion, our primary goal here is to prove that the ill-posedness of the formal system \eqref{Toy:model}, derived from \eqref{EM} in the limit $c\to 0$, persists for small values of $c$. In fact, we will show that the ill-posedness remains valid for any finite value of the speed of light $c\in (0,\infty)$.   
To that end, we seek a solution to \eqref{EM} where the electromagnetic field is given by 
\begin{equation*}
	B(t,x)= (\alpha,\beta,0) +  \big (b(t,x), 0\big ), \qquad   b=(b_1,b_2) \qquad  \text{ and } \quad
	E=\big (0,0,E(t,x)\big ) ,
\end{equation*}
for some $(\alpha,\beta) \in \mathbb R^2\setminus \{ (0,0)\}$ and a scalar unknown functions $b_1,b_2$ and $E$.

 For simplicity, we restrict our analysis here to the case $\alpha=1$ and $\beta=0$. The   general case of non-trivial values of $\beta$ can be treated in exactly the same way. In this case, the equations satisfied by $(u,E,b)$ read as follows (henceforth, we also set $\sigma=1$, for simplicity)
\begin{equation} \label{EM:perturbation:velocity}
	\begin{cases}
		\begin{aligned}
			   &\partial_t u +u \cdot\nabla u = - \nabla p + ( j_b- u_2)  b^\perp    + ( j_b- u_2)\vec{e}_2, &\quad \div u =0,&
			\\
			&\frac{1}{c} \partial_t E - \nabla ^\perp \cdot  b =- j_b + u_2, 
			\\
			 &\frac{1}{c} \partial_t b -\nabla ^\perp  E  = 0 , &\div b = 0,&
			\\
			 &j_b=  cE + u^\perp \cdot b   ,
		\end{aligned}
	\end{cases}
\end{equation}
where $\vec e_j$ denotes the $j^{\text{th}}$ component of canonical orthonormal basis of $\mathbb R^3$, and $\nabla^ \perp = (-\partial_{x_2}, \partial_{x_1}) $.

Therefore, rewriting the preceding  system in the vorticity formulation yields that 
\begin{equation}\label{EM'}
	\begin{cases}
		\begin{aligned}
			   &\partial_t \omega +u \cdot\nabla \omega =    b  \cdot \nabla ( j_b- u_2)     +  \partial_{x_1} j_b+\mathcal R \omega , &\quad \div u =0,&
			\\
			&\frac{1}{c} \partial_t E - \nabla ^\perp \cdot  b +cE  =- u^\perp \cdot b + u_2, 
			\\
			 &\frac{1}{c} \partial_t b -\nabla ^\perp  E  = 0 , &\div b = 0,&
			\\
			 &j_b=  cE + u^\perp \cdot b   ,
		\end{aligned}
	\end{cases}
\end{equation}
where $\mathcal R$ is introduced in \eqref{Riesz:def}.

Thus, the roadmap for establishing the  ill-posedness of \eqref{EM} can be summarized  in following three steps:
\begin{enumerate}
	\item {\em   Local-in-time wellposedness.} In Section \ref{section:existence}, we show  that any smooth initial datum $(u_0,E_0,b_0)$, belonging to the critical space $B^{2}_{2,1}(\mathbb R^2)$, generates  a unique local-in-time solution $(u,E,b)$ of \eqref{EM:perturbation:velocity}, maintaining the same initial regularity over the time of existence.
	\item { \em Low-vs-high order components of Lorentz force.} In Section \ref{section:Lorentz:force}, we analyze the Lorentz force. Specifically, we show that the first and second components of the body force  
	\begin{equation*}
		 b  \cdot \nabla ( j_b- u_2)     +  \partial_{x_1} j_b
	\end{equation*}
	are small in $L^1_t B^{1}_{2,1}$ within a certain time frame, while  the third component $\mathcal R \omega$ remains  of order one in the $L^1_tL^\infty$ space.
	 \item {\em Discontinuity of the solution map in $L^\infty$.} 
Following the approach of \cite{EM20}, in this final step we construct a family of initial data indexed by $N\gg1$, where the initial vorticities $\omega_{0,N}$ are small in $L^\infty$ (roughly of size $1/N$), while  both  $ \mathcal R \omega_{0,N}$  and $\omega_{0,N} $ are of order $N$  in $L^\infty$ and $B^0_{\infty,1}$, respectively. We then show that the vorticity remains of order one at some positive time $t\ll1 $. This result stems from the analysis of the previous step, which allows us to treat the vorticity equation from \eqref{EM'} as a perturbation of \eqref{Euler-Riesz} within a specific time frame. 
\end{enumerate}

Notably, in contrast of the approach in \cite{ah} for establishing  global well-posedness of \eqref{EM} in the Yudovich  class under the {\em Normal Structure}, we do not need to use the dispersive properties of Maxwell's equations, that is Strichartz estimates.  All our estimates here are based on elementary energy-type arguments. This allows the findings of the paper to also hold on the two-dimensional torus without requiring substantial changes to the approach.

{\em Organization of the paper.} We embark with a ``black box'', collecting all the necessary tools that we need in our analysis. Thereafter, Sections \ref{section:existence}, \ref{section:Lorentz:force} and \ref{section:ill-posedness} are devoted to the details of each of the three bullet points from the roadmap of the proof discussed above.

\section{Tool box}\label{section:tool:box}

In this section, we recall some important preliminary results that will be essential for the analysis in the following sections. These results are applicable in any dimension $d\geq 2$, and will serve as building blocks for the subsequent proofs.

\subsection{Functional spaces and basic properties} To maintain clarity and simplicity in the presentation, we will omit the detailed discussion of the definitions and basic properties of the functional spaces used in this work, such as Besov and Sobolev spaces. Instead, we refer the reader to the comprehensive statements in \cite{bcd11} for the complete background on the functional analysis tools employed in this paper. Additionally, for a more concise overview of Besov spaces, we refer to the appendix of \cite{ah}, where the same notations will be adopted in the present work. This will allow us to focus on the main results and avoid redundancy in the exposition.

For the notations, the non-homogeneous Besov spaces will be denoted by 
\begin{equation*}
	B^s_{p,q}(\mathbb R^d),
\end{equation*}
for $s\in \mathbb R$, and $p,q\in [1,\infty]$. Sobolev spaces, corresponding to the case $p=q=2$ in the above, are denoted by 
$$H^s(\mathbb R^d).$$

The approach we will follow in the subsequent sections does not require performing advanced para-differential computations. Instead, the only product law we rely on to estimate the nonlinear terms in the perturbed Euler--Maxwell system is captured by the following lemma. 
\begin{lem}[Product-laws]\label{lemma:product}
	For any smooth functions $f$ and $g$, it holds that 
	\begin{equation*}
		\norm {fg}_{B^s_{p,q}(\mathbb R^d)} \lesssim \norm {f}_{L^\infty(\mathbb R^d)} \norm {g}_{B^s_{p,q} (\mathbb R^d)}+ \norm {g}_{L^\infty (\mathbb R^d)} \norm {f}_{B^s_{p,q}(\mathbb R^d)},
	\end{equation*} 
	for any $s>0$ and $p,q\in [1,\infty]$.
\end{lem}

We skip the proof of this lemma, which can be achieved via standard para-differential calculus. See \cite{bcd11}, for instance. 

\subsection{On Riesz operator and the flow map} We introduce the commutator operator 
\begin{equation*}
	[\mathcal R, \Phi] \omega \bydef \mathcal R(\omega \circ \Phi)- \mathcal R(\omega)\circ \Phi,
\end{equation*}
where $\mathcal R$ is given by \eqref{Riesz:def}, and $\Phi $ is the flow map associated with a divergence-free vector field $u$, i.e., 
\begin{equation*}
	\Phi (t,x)= x + \int_0^t u \left( \tau, \Phi(\tau,x) \right) d\tau , \quad \text{for all } (t,x)\in \mathbb R^+\times \mathbb R^d.
\end{equation*}

The following lemma was proved in \cite{EM20}. It will play a central role in our proofs, later on.
\begin{lem}[\cite{EM20}]\label{lemma:commutator:Riesz}
	It holds that 
	\begin{equation*}
		\norm {[\mathcal R,\Phi] \omega }_{B^{\frac{d}{p}}_{p,1}(\mathbb R^d)} \lesssim \max \left \{ \norm {\Phi-\id }_{\dot W^{1,\infty}(\mathbb R^d)} , \norm {\Phi^{-1}-\id }_{\dot W^{1,\infty}(\mathbb R^d)} \right\} \norm {\omega}_{B^{\frac{d}{p}}_{p,1}(\mathbb R^d)},
	\end{equation*}
	for any $p\in (1,\infty)$, where the implicit constant depends only on  dimension and the Riesz operator.
\end{lem}

The next lemma establishes the continuity in Besov spaces of the composition with the bi-Lipschitz volume-preserving flow $\Phi $ introduced above.

\begin{lem}[\cite{V98}]\label{lemma:flow:Besov}
	It holds that 
	\begin{equation*}
		\norm {f\circ \Phi }_{B^0_{\infty,1}(\mathbb R^d)} \lesssim \left(1 + \log \left( \norm {\Phi}_{\dot W^{1,\infty}(\mathbb R^d)}\norm {\Phi^{-1}}_{\dot W^{1,\infty}(\mathbb R^d)} \right) \right) \norm f_{B^0_{\infty,1} (\mathbb R^d)},
	\end{equation*}
	for any $f\in B^0_{\infty,1}$, where the implicit constant depends only on the dimension.
\end{lem}

\subsection{Transport equation in Besov spaces}

We now state the following classical lemma on Besov estimates for transport equations.  
\begin{lem}[\cite{bcd11}]\label{lemma:transport}
	Let $f$ be a regular solution of the transport equation 
	\begin{equation*}
		\partial_t f + u\cdot\nabla f = g, \qquad f|_{t=0}=f_0,
	\end{equation*}
	on some time interval $[0,t]$, for some smooth divergence-free velocity field $u$. Then it holds that
\begin{equation*}
	\norm {f}_{L^\infty([0,t];B^s_{p,q}(\mathbb R^d))} \leq \norm {f_0}_{B^s_{p,q}(\mathbb R^d)} + C\int_0^t \norm {\nabla u(\tau)}_{L^\infty(\mathbb R^d)} \norm {f(\tau)}_{B^s_{p,q}(\mathbb R^d)} d\tau + \norm {g}_{L^1([0,t];B^s_{p,q} (\mathbb R^d))},
\end{equation*}
for any $s\geq 0$, $p,q\in [1,\infty]$, for some constant $C=C(s,p,q)>0$. 
\end{lem}

\subsection{Maxwell equation in Besov spaces} For the sake of completeness, we conclude by a proof of the next lemma about estimates for the maxwell equation
\begin{equation}\label{Maxwell:equation}
	\begin{cases}
		\begin{aligned}
			&\frac{1}{c} \partial_t E - \nabla \times  b +cE  =F, 
			\\
			 &\frac{1}{c} \partial_t b +\nabla  \times   E  = 0  ,
		\end{aligned}
	\end{cases}
\end{equation}
where $F$ is a  given body  force. 
\begin{lem}\label{lemma:Maxwell}
	Given $c\in (0,\infty)$, and let  $(E,b)$   be a smooth solution of \eqref{Maxwell:equation} supplemented with the inputs 
	$$(E_0,b_0 ) \in B^s_{2,1}(\mathbb R^d) \qquad \text{and} \qquad F\in L^1([0,T];B^s_{2,1}(\mathbb R^d)), $$
	for some $s\in \mathbb R$ and $T>0$. Then, it holds that 
	\begin{equation*}
		\norm {(E,B)}_{L^\infty ([0,t]; B^s_{2,1})} \leq 2\norm  {(E_0,B_0)}_{B^s_{2,1}} + 2c\int_0^t \norm {F(\tau) }_{B^s_{2,1}(\mathbb R^d)} d\tau,
	\end{equation*}
	for any $t\in [0,T]$.
\end{lem}

\begin{proof}
	Localizing the frequencies of \eqref{Maxwell:equation} by applying the dyadic blocs $\Delta_j$ (see \cite{bcd11} for the precise  definition), we find that 
	\begin{equation*} 
	\begin{cases}
		\begin{aligned}
			&\frac{1}{c} \partial_t E_j - \nabla \times  b_j +cE _j =F_j, 
			\\
			 &\frac{1}{c} \partial_t b_j +\nabla  \times   E_j   = 0,
		\end{aligned}
	\end{cases}
\end{equation*}
where we set 
\begin{equation*}
	(E_j, b_j, F_j)\bydef   (\Delta_jE,\Delta_jb,\Delta_jF).
\end{equation*}
Thus, by an $L^2$-energy estimate,   we obtain that 
\begin{equation*}
	\norm {(E_j,b_j)(t)}_{L^2} \frac{d}{dt} \left( \norm {(E_j,b_j)(t)}_{L^2} \right) + c^2 \norm {E_j(t)}_{L^2}^2=c \int_{\mathbb R^d} (F_j \cdot E_j)(t,x) dx \leq \norm {F_j(t)}_{L^2} \norm {E_j(t)}_{L^2} ,
\end{equation*}
for any $t\in [0,T]$. Therefore, it follows that 
\begin{equation*}
	\norm {(E_j,b_j)(t)}_{ L^2} \leq \norm {(E_j,b_j)|_{t=0}}_{L^2}  +   \norm {F_j}_{L^1 ([0,t] ;L^2)}.
\end{equation*}
At last, multiplying both sides by $2^{js}$, taking the sum over $j\geq -1$, and employing the embedding $\ell ^1 \hookrightarrow \ell ^2$ yields the desired control. This completes the proof of the lemma.
\end{proof}

\section{Local-in-time smooth solutions}\label{section:existence}

In this section, we show that the equations \eqref{EM'} around a steady magnetic background admits a unique local-in-time smooth solution. Before delving into the specifics of that, we first establish a formal $L^2$ energy estimate for the perturbation $(u, E, b)$.  

It is important to note that working on the whole space $\mathbb{R}^2$ does not allow us to recover an $L^2$-energy bound for the fluctuation $(u, E, b)$ from the standard energy bound on the original solution $(u, E, B)$ of \eqref{EM}. This is because only one of them (speaking of the magnetic field) can have a finite energy when considering the entire plane $\mathbb{R}^2$. However, it turns out that the perturbation $(u, E, b)$ itself enjoys a similar $L^2$-energy bound, provided that its initial data has a finite $L^2$-norm. 

The following lemma formalizes this result, which in turn plays a crucial role in the proof of the ill-posedness of the system \eqref{EM}. 

\begin{lem}[Energy]\label{lemma:energy}
Any smooth solution of \eqref{EM'} enjoys the bound 
\begin{equation*}
	\norm {(u,E,b)(t)}_{L^2} \leq  \norm {(u_0,E_0,b_0)}_{L^2},
\end{equation*}
	for all $t\geq 0$.
\end{lem}

\begin{proof}
	Taking the $L^2_{t,x}$ scalar product   of the first three equations from \eqref{EM'} with $(u,E,b)$, and making use of Ohm's law, yields that 	\begin{equation*}
	\begin{aligned}
		\frac{1}{2}	\norm {(u,E,b)(t)}_{L^2}^2 + \int_0^t \norm {j_b(\tau)} _{L^2}^2 d\tau &+ \int_0^t \norm {u_2(\tau)} _{L^2}^2 d\tau 
		\\
		&= \frac{1}{2} \norm {(u_0,E_0,b_0)}_{L^2}^2 + \int_0^t \int_{\mathbb R^2} (cE+j_b)u_2  -  \int_0^t \int_{\mathbb R^2} u_2 b^\perp \cdot u
		\\
		&= \frac{1}{2} \norm {(u_0,E_0,b_0)}_{L^2}^2 + 2\int_0^t \int_{\mathbb R^2} j_bu_2   
		\\
		&\leq  \frac{1}{2} \norm {(u_0,E_0,b_0)}_{L^2}^2 +\int_0^t \norm {j_b(\tau)} _{L^2}^2 d\tau + \int_0^t \norm {u_2(\tau)} _{L^2}^2 d\tau ,
	\end{aligned}
	\end{equation*}
	for any $t\geq 0$, where we used the basic Young's inequality. Finally, absorbing the  last two terms on the right-hand side by the same quantity from the left-hand side completes the proof.
\end{proof}

Now, we prove the main result of this section, that is the local-in-time well-posedness of \eqref{EM'} in the Besov space $B^2_{2,1}(\mathbb R^2)$.  

\begin{thm}[Local smooth solutions]\label{thm2-EM}
	For any initial datum
	\begin{equation*}
		(u _0,E_0,b_0)\in B^{2}_{2,1}   (\mathbb {R}^2) ,
	\end{equation*}
	there is a unique local-in-time solution of \eqref{EM'}  enjoying the bound
	\begin{equation*}
		(u,E,b)\in C^0\big([0,t];B^2_{2,1}(\mathbb R^2)\big)\cap C^1\big([0,t];B^1_{2,1}(\mathbb R^2)\big),
	\end{equation*}
	for some finite $t>0$.
	More precisely,  there is a universal constant  $C_*>0$ such that  
	\begin{equation*}
		\norm {(u,E,b)(t)}_{  B^{2}_{2,1}   (\mathbb {R}^2))} \leq 4 \norm {(u_0,E_0,b_0)}_{B^{2}_{2,1}  (\mathbb {R}^2)},
	\end{equation*}
	for any positive time $t \leq \frac{1}{C_*(1+c^2+ \mathcal E_0)}$  satisfying 
	\begin{equation*}
		t \norm {(u_0,E_0,b_0)}_{B^{2}_{2,1}  (\mathbb {R}^2)}\leq  \frac{ 1}{  4C_*(1+c^2 + \mathcal E_0)},
	\end{equation*}
	where 
	\begin{equation*}
	\mathcal	E_0 \bydef \norm {(u_0,E_0,b_0)}_{L^2} . 
	\end{equation*}
	\end{thm}
 
 \begin{proof} We will only focus our attention on obtaining the necessary a priori estimate, assuming that one has already constructed a smooth approximate solution of \eqref{EM'}. The rigorous justification of the existence of solutions can be done via a routine approximation and compactness procedure. See for instance \cite{ah} for an admissible scheme of \eqref{EM} that applies to \eqref{EM'} as well.
 
   It  is also more convenient to perform the computations on the vorticity formulation. Thus, we begin by applying Lemma \ref{lemma:transport} to obtain that 
 	\begin{equation*} 
 			\begin{aligned}
 				\norm {\omega(t)}_{B^{1}_{2,1}  } 
 		&	\lesssim  	\norm {\omega_0}_{B^{1}_{2,1} }+ \int_0^t \left(  \norm {\nabla u(\tau )}_{L^\infty}  + 1\right)	\norm {\omega(\tau )}_{ B^{1}_{2,1} } d\tau 
 		\\
 		&\quad + \norm {b\cdot \nabla  u_2}_{L^1_t B^{1}_{2,1} }   + \norm { b \cdot \nabla j_b}_{L^1_t B^{1}_{2,1} } + \norm {\nabla j _b}_{L^1_t B^{1}_{2,1} }
 		\\
 		&	\lesssim  	\norm {\omega_0}_{B^{1}_{2,1} }+ t(1+ \norm \omega_{L^\infty_t B^1_{2,1}}) \norm \omega_{L^\infty_t B^1_{2,1}} 
 		\\
 		& \quad + \norm {b\cdot \nabla  u_2}_{L^1_t B^{1}_{2,1} }   + \norm { b \cdot \nabla j_b}_{L^1_t B^{1}_{2,1} } + \norm {\nabla j_b }_{L^1_t B^{1}_{2,1} },
 			\end{aligned}
 			\end{equation*}
 			for any $t\geq 0$, where we used the embedding 
 			\begin{equation*}
 				B^1_{2,1}\hookrightarrow L^\infty (\mathbb R^2)
 			\end{equation*}
 			together with the fact that Riesz's operator is bounded on Besov spaces.
 			 Therefore, applying the product laws from Lemma \ref{lemma:product} yields that 
 			\begin{equation*}
 			\begin{aligned}
 				\norm {\omega(t)}_{B^{1}_{2,1}  }  
 			&\lesssim  	\norm {\omega_0}_{B^{1}_{2,1}  }+ t(1+ \norm { b }_{L^\infty _{t} B^1_{2,1} }+  \norm \omega_{L^\infty_t B^1_{2,1}}) \norm \omega_{L^\infty_t B^1_{2,1}} 
 			\\ 
 			& \quad + t(1+ \norm { b }_{L^\infty _{t} B^1_{2,1} })\norm { \nabla j_b}_{L^\infty _tB^{1}_{2,1} }.  
 			\end{aligned}
 			\end{equation*}
 			
 		Next, expending the expression of $j_b$ by employing the definition of Ohm's law, and applying the product laws from Lemma \ref{lemma:product} to estimate the resulting terms leads to
 		\begin{equation*}
 			\begin{aligned}
 				\norm {\omega(t)}_{B^{1}_{2,1}  }  
 			&\lesssim  	\norm {\omega_0}_{B^{1}_{2,1}  }+ t(1+ \norm { b }_{L^\infty _{t} B^1_{2,1} }+  \norm \omega_{L^\infty_t B^1_{2,1}}) \norm \omega_{L^\infty_t B^1_{2,1}} 
 			\\ & \quad + t(1+ \norm { b }_{L^\infty _{t} B^1_{2,1} })\left( \norm { c\nabla E }_{L^\infty _tB^{1}_{2,1} } + \norm { \nabla ( u^\perp \cdot b)}_{L^\infty _tB^{1}_{2,1} } \right)
 			\\
 			&\lesssim  	\norm {\omega_0}_{B^{1}_{2,1}  }+ t(1+ \norm { b }_{L^\infty _{t} B^1_{2,1} }+  \norm \omega_{L^\infty_t B^1_{2,1}}) \norm \omega_{L^\infty_t B^1_{2,1}} 
 			\\ & \quad + t(1+ \norm { b }_{L^\infty _{t} B^1_{2,1} })\left( \norm { c\nabla E }_{L^\infty _tB^{1}_{2,1} } + \norm { u}_{L^\infty _tB^{1}_{2,1} } \norm b_{B^2_{2,1}} +  \norm { b}_{L^\infty _tB^{1}_{2,1} } \norm \omega _{B^1_{2,1}} \right) .
 			\end{aligned}
 			\end{equation*}
 				Hence, by further employing the interpolation inequalities 
 	\begin{equation}\label{interepolation:u:b}
 		\norm b_{L^\infty_{t} B^1_{2,1}} \lesssim \norm b_{L^\infty_t L^2}^\frac{1}{2} \norm { b}_{B^2_{2,1}}^\frac{1}{2} \qquad \text{and} \qquad \norm u_{L^\infty_{t} B^1_{2,1}} \lesssim \norm u_{L^\infty_t L^2}^\frac{1}{2} \norm {\omega}_{B^1_{2,1}}^\frac{1}{2},
 	\end{equation}
 	and the energy bound from Lemma \ref{lemma:energy}, we infer that   
 	   \begin{equation}\label{bound:omega}
 			\begin{aligned}
 				\norm {\omega(t)}_{B^{1}_{2,1}  }   
 			&\lesssim  	\norm {\omega_0}_{B^{1}_{2,1}  }+ t\left (1+\mathcal E_0^\frac{1}{2} \norm { b}_{B^2_{2,1}}^\frac{1}{2}+  \norm \omega_{L^\infty_t B^1_{2,1}}\right ) \norm \omega_{L^\infty_t B^1_{2,1}} 
 			\\ & \quad + t\left (1+ \mathcal E_0^\frac{1}{2} \norm { b}_{B^2_{2,1}}^\frac{1}{2}\right )\left( \norm { cE }_{L^\infty _tB^{2}_{2,1} } + \mathcal E_0^\frac{1}{2} \norm {\omega }_{B^1_{2,1}}^\frac{1}{2} \norm b_{B^2_{2,1}} + \mathcal E_0^\frac{1}{2} \norm { b}_{B^2_{2,1}}^\frac{1}{2}\norm \omega _{B^1_{2,1}} \right) .
 			\end{aligned}
 			\end{equation}

 	Now, we move on to estimate $(E,b)$ in $L^\infty _t B^2_{2,1}$. To that end,  applying Lemma \ref{lemma:Maxwell} yields that 
 	\begin{equation}\label{bound:Eb}
 		\begin{aligned}
 			 \norm {(cE,b)}_{L^\infty_t B^{2}_{2,1}}  
 			&\leq    2\norm {(cE_0,b_0)}_{B^{2}_{2,1}} + 2c^2\norm {u^\perp \cdot b}_{L^1_t B^2_{2,1}} + c^2 \norm {u_2}_{L^1_t B^2_{2,1}}
 			\\
 			&\leq  2 \norm {(cE_0,b_0)}_{B^{2}_{2,1}} + C \left( c^2 t \norm {\omega }_{L^\infty _t B^1_{2,1}}\left( 1+\norm {  b}_{L^\infty _t B^2_{2,1}} \right) \right),
 		\end{aligned}
 	\end{equation}
 	for some universal constant $C>0$,
 	where Lemma \ref{lemma:product} has been used, once again.
 	 	
 	All in all, introducing the time-functional 
 	\begin{equation*}
 		f(t) \bydef \norm {\omega}_{L^\infty_t B^1_{2,1}} + \norm {(E,b)}_{L^\infty_t B^2_{2,1}} ,
 	\end{equation*}
 	 and combining \eqref{bound:omega} and \eqref{bound:Eb}, we end up with the control  
 	\begin{equation}\label{f:bound:existence}
 		f(t) \leq  2 f_0 + Ct (1+c^2+ \mathcal E_0) f(t)+ Ct(1+c^2 + \mathcal E_0  ) f^2(t) ,
 	\end{equation}
 	for up to possibly suitably change the constant $C>0$, and for any $t\geq 0$.
 	
 		Now, defining 
 	\begin{equation*} 
 		T\bydef \sup  \left \{ t \in \left (0, \frac{1}{4C(1+c^2 + \mathcal E_0)} \right) : tf(t) < \frac{ 1}{  4C(1+c^2+ \mathcal E_0)}  \right \}
 	\end{equation*}
 	yields that 
 	\begin{equation}\label{f:bound}
 		f(t) \leq 4 f_0, \quad \text{for all} \quad t\leq T. 
 	\end{equation}
 	To conclude, if one imposes the condition that 
 	\begin{equation*}
 		tf_0 < \frac{ 1}{  16C(1+c^2 + \mathcal E_0)},
 	\end{equation*}
 	then it follows, by a standard continuity argument, that the bound \eqref{f:bound} would be valid for any positive  time satisfying 
 	$$ t \leq \frac{1}{4C (1+c^2+ \mathcal E_0)} \quad \text{and} \quad  tf_0 < \frac{ 1}{  16C(1+c^2 + \mathcal E_0)}.$$
 	This concludes the proof of the a priory bound on the solution, which implies the existence of solution via a usual approximation and compactness procedure.

 	The uniqueness component of the statement can be recovered  along the same lines as it was previously done for the original Euler--Maxwell system in \cite{ah}. We skip the details of that, thereby completing the proof of the theorem.
 \end{proof}
 
 \begin{rem}
 	Notice that we have paid a particular attention to the estimate of the vorticity   by using the interpolation inequalities \eqref{interepolation:u:b}, 	and eventually employing the energy bound from Lemma   \ref{lemma:energy}, rather than naively using the embedding
 	\begin{equation}\label{embedding_nonhom}
 		B^2_{2,1} \hookrightarrow B^1_{2,1}(\mathbb R^2).
 	\end{equation}
 	 This procedure is motivated by the aim of obtaining a quadratic nonlinear estimate for the time functional $f(t)$, see \eqref{f:bound:existence}. Indeed, the naive aforementioned embedding would lead to a cubic type of the same bound on $f(t)$. The benefit of all of that is to end up with a time of existence that is of order one, instead of a time of order zero that would be suggested by the naive approach of using the non-homogeneous embedding \eqref{embedding_nonhom}. See also the remark after Corollary \ref{corollary:Eb:bound} where the scaling of the initial data that will be used in the proof of Theorem \ref{main:thm} is   briefly discussed. 
 \end{rem}
 
 \section{Low-vs-high order components of Lorentz force around a   magnetic background}\label{section:Lorentz:force}
 
 In this section, we perform a simple analysis that allows us to compare the different terms from Lorentz force, after expending the magnetic field around the steady state $B=(1,0,0)$, for instance. As a result, we will eventually  identify two essential behaviors: The dominant Riesz part $\mathcal R \omega$, which will be of order one  within a specific time-space scaling, and the remaining ``negligeable''      terms from Lorentz force, which will be of order zero within the same preceding scaling. In particular,  this is the main observation that allows us to conclude that the vorticity of Euler--Maxwell with a magnetic field around an horizontal steady background is essentially driven by the equation  
 \begin{equation*}
 	\partial_t \omega + u\cdot \nabla \omega = \mathcal R \omega + o (1), 
 \end{equation*} 
 within a specific time interval. 
 Consequently, the mild ill-posedness follows directly by adapting the proof from \cite{EM20}, as we shall see in the next section.
 
  We have chosen to present the content of this part of the paper in a sort of ``independent'' section as we believe that it will be very useful in an upcoming work on a relevant problem about the Euler--Maxwell system.
 
 Let us get started by the setup.
 In the remaining sections, we will assume that 
 \begin{equation}\label{assumption:data}
 	\norm {(cE_0,b_0)}_{B^2_{2,1}}  \leq  \norm {u_0}_{B^2_{2,1}}
 \end{equation}
 holds. Accordingly, we
  denote $T^*$ the positive time given by 
 \begin{equation*} 
 	T^*\bydef \frac{1}{C_*(1+c^2+ \mathcal E_0)}\min\left \{ 1, \frac{1}{2\norm {u_0}_{B^2_{2,1}} } \right\},
 \end{equation*} 
 which will correspond to the time of existence of the strong solution given by Theorem \ref{thm2-EM}. Moreover, by virtue of the same theorem, we have that 
 \begin{equation}\label{bound:omega:2}
 	\norm {\omega}_{L^\infty_t B^1_{2,1}} \lesssim   \norm {\omega_0}_{B^1_{2,1}},
 \end{equation} 
 for any $t\leq T^*$.
 
 As a first direct consequence of Theorem \ref{thm2-EM}, we have the following  corollary.
 \begin{cor}\label{corollary:Eb:bound} Assuming \eqref{assumption:data}, it then  holds that the solution of \eqref{EM'} established in Theorem \ref{thm2-EM} enjoys the bound   
 	\begin{equation*}
 		\norm {(cE,b)}_{L^p_t B^2_{2,1}} \lesssim  t^{\frac{1}{p}} \norm {(cE_0,b_0)}_{B^2_{2,1}} + c^2 t^{1+\frac{1}{p}} \norm {\omega_0}_{B^1_{2,1}},
 	\end{equation*}
 	for any $t \leq  T^*$ and $p\in [1,\infty]$.
 \end{cor}  

\begin{proof}
	This follows directly from the bound  on the local solution obtained by Theorem \ref{thm2-EM}, which is in turn  employed in the control of the electromagnetic fields given by \eqref{bound:Eb}, under the assumption on the size of the initial data \eqref{assumption:data}. The proof of the corollary is now complete.
\end{proof}
\begin{rem}
	Later on, in the proof of Theorem \ref{main:thm}, we will be considering a family of initial data that exhibits the time-space scaling   
		\begin{equation*}
		t \cdot \norm {\omega_{0,N}}_{B^1_{2,1}}  \sim \varepsilon t N \sim 1  \quad \text{and} \quad  \norm {(cE_{0,N},b_{0,N})}_{B^2_{2,1}} \sim o (\frac{1}{N}), \quad N , \varepsilon^{-1}\gg1.
	\end{equation*}
	 Corollary \ref{corollary:Eb:bound} establishes that, within the preceding scaling,  $L^\infty_t B^{2}_{2,1}$ norms of the electromagnetic fields are of order one, while $L^p_t B^{2}_{2,1}$ norms, for any finite exponent ``$p$'', are of order zero. Fortunately, all the residual terms resulting from the Lorentz force that appear in   writing  Euler--Maxwell equations around the horizontal magnetic background will be negligible in comparaison to the term ``$\mathcal R \omega$'' due to this simple observation.
\end{rem}

In the next proposition, we  recast  the core estimate   that can be derived from the a priori estimate in proof of Theorem \ref{thm2-EM}, yielding a control over the difference between the Lorentz force and mere    partial space derivatives of the velocity field, for a magnetic field around an horizontal steady background.
\begin{prop}[Lorentz force around a magnetic background]\label{prop:source:term}
	Let $(u,E,B)$ be the smooth solution of \eqref{EM} defined on $[0,t]$, for some positive time $t>0$, where 
	\begin{equation*}
		B=(1,0,0) + (b,0), \quad \text{with} \quad b=(b_1,b_2).  
	\end{equation*} 
	Then, it holds that 
	\begin{equation*}
		\norm {\nabla \times (j\times B) - \mathcal R \omega}_{L^1([0,t]; B^1_{2,1}(\mathbb R^2))} \lesssim \left( \mathcal E_0  \norm {\omega_0}_{B^1_{2,1}}+1 \right)\left( \norm {(cE_0,b_0)}_{B^2_{2,1}} +  c^2 t  \norm {\omega_0}_{B^1_{2,1}} \right) t  .
	\end{equation*}
\end{prop} 

\begin{rem}
	Once again, according to the scaling of the family of the initial data that we will be using in the proof of Theorem \ref{main:thm} later on, it is actually extremly important to interpolate with the energy bound from Lemma \ref{lemma:energy} and have the factor $``\mathcal E_0$'' for the first term in the right-hand side. This guarantees that, within the time scaling regime discussed just after Corollary \ref{corollary:Eb:bound}, the term 
	\begin{equation*}
		\mathcal E_0 (t\norm {\omega_0}_{B^1_{2,1}})^2
	\end{equation*}
	has a zero order, which is very crucial.
\end{rem}
\begin{proof}
	It readily seen that 
	\begin{equation*}
		\nabla \times (j\times B) - \mathcal R \omega = b  \cdot \nabla ( j_b- u_2)     +  \partial_{x_1} j _b ,
	\end{equation*}
	where 
	\begin{equation*}
		j_b=cE + u^\perp \cdot b.
	\end{equation*}
	Therefore, by employing the fact that $B^1_{2,1}(\mathbb R^2) $ is an algebra for the product operation, we find that 	\begin{equation*}
	\begin{aligned}
		& \norm {b\cdot \nabla  u_2}_{L^1_t B^{1}_{2,1} }   + \norm { b \cdot \nabla j_b}_{L^1_t B^{1}_{2,1} } + \norm {\partial_{x_1} j _b}_{L^1_t B^{1}_{2,1} } 
		\\
		& \qquad \qquad\lesssim 
		(1+ \norm { b }_{L^\infty _{t} B^1_{2,1} })\left( c \norm { \nabla E}_{L^1_t B^{1}_{2,1} } +  \norm {  u  }_{L^\infty _{t} B^{1}_{2,1}  }\norm { \nabla   b }_{L^1_t B^{1}_{2,1} } +  t\norm {  b  }_{L^\infty _{t} B^{1}_{2,1}  }\norm { \omega }_{L^\infty _t B^{1}_{2,1} }\right) ,
	\end{aligned}	
		\end{equation*}
		which yields, by virtue of the interpolation inequalities
		\begin{equation*}
 		\norm b_{L^\infty_{t} B^1_{2,1}} \lesssim \norm b_{L^\infty_t L^2}^\frac{1}{2} \norm { b}_{L^\infty_t  B^2_{2,1}}^\frac{1}{2} \qquad \text{and} \qquad \norm u_{L^\infty_{t} B^1_{2,1}} \lesssim \norm u_{L^\infty_t L^2}^\frac{1}{2} \norm {\omega}_{L^\infty_{t} B^1_{2,1}}^\frac{1}{2},
 	\end{equation*}
		combined with the energy bound from Lemma \ref{lemma:energy} and the control \eqref{bound:omega:2},  that
		\begin{equation*}
	\begin{aligned}
		\norm {b\cdot \nabla  u_2}_{L^1_t B^{1}_{2,1} } &   + \norm { b \cdot \nabla j_b}_{L^1_t B^{1}_{2,1} } + \norm {\nabla j _b}_{L^1_t B^{1}_{2,1} } 
		\\
		&\lesssim 
		(1+ \mathcal E_0^\frac{1}{2} \norm {\omega_0}_{B^1_{2,1}}^\frac{1}{2} )\Bigg ( \left( \norm {(cE_0,b_0)}_{B^2_{2,1}} +  c^2 t  \norm {\omega_0}_{B^1_{2,1}} \right)t 
		\\
		& \quad +  \mathcal E_0^\frac{1}{2} \norm {\omega_0}_{B^1_{2,1}}^\frac{1}{2}\left( \norm {(cE_0,b_0)}_{B^2_{2,1}} +  c^2 t  \norm {\omega_0}_{B^1_{2,1}} \right) t \Bigg ) 
		\\
		& \lesssim \left( \mathcal E_0  \norm {\omega_0}_{B^1_{2,1}}+1 \right)\left( \norm {(cE_0,b_0)}_{B^2_{2,1}} +  c^2 t  \norm {\omega_0}_{B^1_{2,1}} \right) t   .
	\end{aligned}	
		\end{equation*}
		This  completes the proof of the proposition.
\end{proof}

\section{Ill-posedness}\label{section:ill-posedness}

This section is devoted to the proof of the main result of the paper, that is the ill-posedness of the Euler--Maxwell system \eqref{EM} in the Yudovich class when the magnetic field is nearby an horizontal magnetic background. We embark with the following lemma (see  \cite[Proposition 1.5]{WZ23}, for instance)
\begin{lem}\label{lemma:data}
	There is a smooth family of functions $(f_N)_{N\in \mathbb N^*}$ with the properties that 
	\begin{equation*}
		\norm {f_N}_{\dot H^{-1} \cap L^2\cap L^\infty } = \mathcal {O}(1),
	\end{equation*}
	whereas
	\begin{equation*}
		\norm {\mathcal R f_N}_{ L^\infty}  \gtrsim N
	\end{equation*}
	and 
	\begin{equation*}
		\norm {f_N}_{  B^\frac{2}{p}_{p,1} } \lesssim  N,
	\end{equation*}
	for any $p\in (1,\infty)$. 
\end{lem}

Now, we are in a position to prove Theorem \ref{main:thm}.
\begin{proof}[Proof of Theorem \ref{main:thm}] We supply the Euler--Maxwell system with the initial datum
\begin{equation*}
	(\omega_{0,N},E_{0,N}, B_{0,N})= \varepsilon (f_N, g\vec e_3, \vec e_1+ \nabla^\perp    g ), \quad \text{for all } N\in \mathbb N^*,
\end{equation*}
where $(f_N)_{N\in \mathbb N^*}$ is the family of functions introduced in Lemma \ref{lemma:data}, and $g \in C^\infty_c (\mathbb R^2)$ is any fixed smooth profile. Therefore, Theorem \ref{thm2-EM} gives rise to a unique local-in-time solution of \eqref{EM} defined on $[0,T^*_N]$, where we can set
\begin{equation*}
	T_N^* \bydef \frac{1}{C_* (2+c^2)\varepsilon N},
\end{equation*}
for some suitable contant $C>0$.
 Moreover, the vorticity enjoys the  bound 
\begin{equation}\label{bound:solution:main:proof}
	\norm {\omega}_{L^\infty_t B^1_{2,1}} \lesssim  \norm {\omega_{0,N}}_{B^1_{2,1}} \lesssim \varepsilon N ,
\end{equation}
for any $t\leq T_N^*$.
 Now,  we recast the equation of the vorticity  \begin{equation*}
	\partial_t \omega +u \cdot\nabla \omega =  \mathcal R \omega +  b  \cdot \nabla ( j_b- u_2)     +  \partial_{x_1} j_b
\end{equation*}
	as
	\begin{equation*}
 	(\omega \circ  \Phi ) (t) = e^{t \mathcal R} \omega_0  + \int_0^t e^{(t-\tau) \mathcal R}   [\mathcal R, \Phi ] \omega (\tau)  d\tau + \int_0^t e^{(t-\tau) \mathcal R}  (F\circ \Phi ) (\tau)   d\tau ,
 \end{equation*}
 where $\Phi $ denotes the flow associated with the divergence-free vector field $u$, and 
 \begin{equation*}
 	F\bydef  b  \cdot \nabla ( j_b- u_2)     +  \partial_{x_1} j_b.
 \end{equation*}
 Thus, it follows, for any $t\in [0,T^*_N]$, that 
 \begin{equation}\label{lower_bound}
 	\norm {\omega(t)}_{L^\infty} \geq \norm  {e^{t \mathcal R} \omega_0}_{L^\infty} - \norm { \int_0^t e^{(t-\tau) \mathcal R}   [\mathcal R, \Phi ] \omega (\tau)  d\tau}_{L^\infty} - \norm {\int_0^t e^{(t-\tau) \mathcal R}  (F\circ \Phi ) (\tau)   d\tau}_{L^\infty}.
 \end{equation}
 To estimate the first term, we follow \cite{EM20} by writing that   
 \begin{equation*}
 	e^{t \mathcal R} \omega_0 = \omega_0 + t\mathcal R\omega_0 + t^2 \left(  \sum_{n\geq 2} \frac{t^{n-2} \mathcal R^n}{n!} \right) \omega_0.
 \end{equation*}
 Thus, using that $\mathcal R$ is bounded on $B^1_{2,1}(\mathbb R^2)$, we infer that 
 \begin{equation*} 
 	\norm {e^{t \mathcal R} \omega_0 }_{L^\infty} \geq t \norm {\mathcal R \omega_0}_{L^\infty} - \norm {\omega_0}_{L^\infty} - C t^2 \norm {\omega_0}_{B^1_{2,1}}.
 \end{equation*}
 By incorporating the scaling of the initial data, it follows that 
 \begin{equation}\label{first_term}
 	\norm {e^{t \mathcal R} \omega_0 }_{L^\infty} \geq C\varepsilon t N  - \varepsilon - C t^2 \varepsilon N ,
 \end{equation}
 for some suitable constant $C>0$.
 For the second term, we employ the commutator estimate from Lemma \ref{lemma:commutator:Riesz} to find that   \begin{equation*}
 	\begin{aligned}
 		\norm { \int_0^t e^{(t-\tau) \mathcal R}   [\mathcal R, \Phi ] \omega (\tau)  d\tau}_{L^\infty} &\leq C t^2 \norm {\nabla u}_{L^\infty_{t,x}} e^{t \norm {\nabla u}_{L^\infty_{t,x}}}   \norm{\omega}_{L^\infty_t B^{1}_{2,1}}
 		\\
 		&\leq C t^2 \norm{\omega}_{L^\infty_t B^{1}_{2,1}}^2 e^{Ct \norm{\omega}_{L^\infty_t B^{1}_{2,1}}}.  
 	\end{aligned}
 \end{equation*}
 By further employing \eqref{bound:solution:main:proof}, we end up with the control 
 \begin{equation} \label{second_term}
 	\norm { \int_0^t e^{(t-\tau) \mathcal R}   [\mathcal R, \Phi ] \omega (\tau)  d\tau}_{L^\infty}  \leq C \left ( \varepsilon t N\right)^2 e^{C\varepsilon t N}.
 \end{equation}
 As for the third term, we have, due to Vishik's Lemma \ref{lemma:flow:Besov}, that 
 \begin{equation*}
 	\begin{aligned}
 		\norm {\int_0^t e^{(t-\tau) \mathcal R}  (F\circ \Phi ) (\tau)   d\tau}_{B^0_{\infty,1}}
 		& \lesssim \int_0^t \norm{ (F\circ \Phi )(\tau)}_{B^0_{\infty,1}} d\tau 
 		\\
 		&\lesssim   \int_0^t \left(1+  \log\left ( \norm{  \Phi(\tau )  }_{\dot W^{1,\infty}} \norm{  \Phi^{-1}(\tau )  }_{\dot W^{1,\infty}} \right) \right)  \norm{F (\tau)}_{B^0_{\infty,1}} d\tau .
 	\end{aligned} 
 \end{equation*}
 Next, using the bound
 \begin{equation*}
 	\norm {\Phi (\tau)}_{\dot W^{1,\infty}} \norm {\Phi^{-1} (\tau)}_{\dot W^{1,\infty}}  \leq e^{2\tau  \norm {\nabla u(\tau)}_{L^\infty}},
 \end{equation*}
 we find that 
 \begin{equation*}
 	\begin{aligned}
 		\norm {\int_0^t e^{(t-\tau) \mathcal R}  (F\circ \Phi ) (\tau)   d\tau}_{B^0_{\infty,1}}
 		& \lesssim    \int_0^t \left(1+ \tau \norm {\nabla u(\tau)}_{L^\infty} \right)    \norm{F (\tau)}_{B^0_{\infty,1}} d\tau 
 		\\
 		& \lesssim  \left(1+ t \norm {\omega}_{L^\infty_t B^1_{2,1}} \right)    \norm{F }_{L^1_tB^0_{\infty,1}}.
 	\end{aligned} 
 \end{equation*}
 Therefore, by applying \eqref{bound:solution:main:proof} and the bound from Proposition \ref{prop:source:term}, we deduce that 
 \begin{equation}\label{third_term}
 	\norm {\int_0^t e^{(t-\tau) \mathcal R}  (F\circ \Phi ) (\tau)   d\tau}_{B^0_{\infty,1}}
 		\lesssim \left(1+ \varepsilon tN \right)  \left( \varepsilon^2 N +1 \right)\left( \varepsilon +  c^2 \varepsilon t  N \right) t.
 \end{equation}
 
 All in all, inserting \eqref{first_term}, \eqref{second_term} and \eqref{third_term} into \eqref{lower_bound}, and choosing the parameters $(\varepsilon,N,t)$ in such away that
 \begin{equation*}
 	\varepsilon \ll 1, \quad t\ll 1, \qquad N\gg1,
 \end{equation*}
with 
\begin{equation*}
	\varepsilon t N  \in (0,1) \quad  (\text{very small but not vanishing}), 
\end{equation*}
	yields the desired lower bound 
	\begin{equation*}
		\norm {\omega (t)}_{L^\infty} \geq c_* >0,
	\end{equation*}
	for some $c_*$ that depends only the speed of light $c$. This completes the proof of the main theorem.
	
\end{proof}



\bibliographystyle{plain} 
\bibliography{plasma}

\end{document}